\documentclass[12ptt]{article}
\usepackage{graphicx}
\usepackage{amssymb}
\usepackage{amsmath}
\usepackage{amsthm}
\newtheorem{thm}{\protect\theoremname}

\newtheorem{lem}[thm]{\protect\lemmaname}

\newtheorem{prop}[thm]{\protect\propositionname}

\newtheorem{cor}[thm]{\protect\corollaryname}

\newtheorem{dfn}[thm]{\protect\definitionname}

\makeatother

\providecommand{\lemmaname}{Lemma}
\providecommand{\propositionname}{Proposition}
\providecommand{\theoremname}{Theorem}
\providecommand{\corollaryname}{Corollary}
\providecommand{\definitionname}{Definition}

\def\ol{\overline}
\newcommand{\gph}{\mathrm{Gr}\,}
\def\s{\mathbf{s}}

\newcommand{\cl}{\mathrm{cl}\,}

\def\R{\mathbb{R}}
\def\N{\mathbb{N}}
\def\eps{\varepsilon}
\def\la{\lambda}
\def\grf{\mathrm{Gr}\,}
\def\supp{\mathrm{supp}\,}

\begin{document}

\title{Surjectivity in Fr\'echet Spaces}

%\subtitle{Using  the  LaTex Template}

\author{Milen Ivanov\footnote{Radiant Life Technologies, milen@radiant-life-technologies.com} \and  Nadia Zlateva\footnote{Faculty of Mathematics and Informatics, 
		St. Kliment Ohridski University of Sofia, 
		5, James Bourchier blvd., 1164 Sofia, Bulgaria, 
		zlateva@fmi.uni-sofia.bg}}

\date{January, 2019}
%The correct dates will be entered by the editor.

\maketitle

\begin{abstract}
We prove surjectivity result in Fr\'echet spaces of Nash-Moser type. That is, with uniform estimates over all seminorms.
	
	Our method works for functions, which are only continuous and strongly G\^ateaux differentiable.
	
	We present the results in multi-valued setting exploring the relevant notions of map regularity.
	
	The key to our method is in geometrising the tameness estimates and thus reducing the problem to a spectrum of problems on suitable Banach spaces. For solving the latter problems we employ an abstract iteration scheme developed by the authors.
\end{abstract}


\thispagestyle{empty}
\newpage

%All acknowledgements should be placed in the back of the paper after Conclusions..

\section{Introduction}

 Theorem of Nash and Moser is a useful tool for studying solvability of certain nonlinear problems with infinitely smooth data. If the corresponding local linear problem is solvable and  some uniform estimates are available, then the original problem is also solvable. In this setting, since the spaces of infinitely smooth functions are not Banach, the usual Inverse Theorems may not work. The monograph \cite{pdonmt} contains a good overview of the subject.
	
	The  proof of Nash-Moser Theorem relies on Newton method to overcome the so called \textit{loss of derivatives}. Thus, it works for twice smooth functions and with  additional assumptions; see \cite{hamilton}.
	
	Often, the injectivity part of Nash-Moser Inverse Theorem is not interesting for applications. What is important is to have some solution for a given right-hand side, that is, surjectivity. See, for example, \cite[p.145]{pdonmt}.
	
	The work \cite{ekeland}, which is one of the inspirations for the present one, shows that certain surjectivity can be proved for functions, which are only G\^ateaux differentiable. The method used is an application of Ekeland Variational Principle to suitably constructed Banach space resembling $l_1$. However, the estimate obtained there is for one chosen norm of the Fr\'echet space (more precisely, for linear combination of the norms, but with non-canonic coefficients), while in the original Nash-Moser Theorem all seminorms are estimated.
	
	The results of \cite{ekeland} were further extended -- with different proofs, but using some of the ideas -- to the more difficult case of non-autonomous \emph{tameness estimates} in \cite{esere}. An application to a classical problem is given there.
	
	Also, because the method of \cite{ekeland} does not require second derivatives, it can be extended to multi-valued mappings; see, e.g., \cite{ngth}.
	
	Here, we further the above development by proving surjectivity result for multi-valued mappings with estimates for all semi-norms. It readily renders to the case of strongly G\^ateaux differentiable function, see Theorem~\ref{thm:ci} below. In place of G\^ateaux differential we use a slight modification of the \textit{contingent variation} defined in \cite{frankovska}.

	The key to our method is considering suitable Banach spaces (resembling $l_\infty$) that fit the structure of the problem. In other words, we geometrize the tameness estimates through the parallelograms $\Pi_{\mathbf{s}}$, see (\ref{eq:def:pi}), which then serve as unit balls of Banach spaces.
	
	Working with multi-valued maps, we highlight the relationship of Nash-Moser Theorem to one of the central concepts in Variational Analysis, namely metric regularity; see, e.g., \cite{Ioffe_book}. Of course, such relationship is already shown in \cite{ngth}. In short, we show that the standard assumptions of Nash-Moser Theorem imply a kind of weak metric regularity of the map. It is interesting that the latter in turn implies metric regularity in one of the possible metrics for the Fr\'echet space.
	
	In the important case of the spaces of infinitely smooth functions, the conclusions are more tight: in fact, as tight as might be reasonably expected, for example we get regularity instead of weak regularity.
	
	It is a feature of our method that it makes crucial the use of the abstract iteration scheme developed in \cite{LOEV}.
	
	However,  non-autonomous tameness estimates, as in \cite{esere}, will be subject of future work.
	
	The reader may get some immediate flavour of the results presented here by looking at Theorem~\ref{thm:ci} and comments there. Theorem~\ref{thm:ci} is fashioned after \cite[Theorem~1]{ekeland}, so it is easy to compare the result to what is already established.

	The  article is organized as follows.
	
	In the next section, we give the precise statements of all our main results. In Section~\ref{sec:prelim}, we establish the technical tools we will be using. In Section~\ref{sec:loev}, we recall the iteration  method we use for approximately solving the equations. Finally, the  proofs are given in Section~\ref{sec:proofs}.

\section{Statements of the Main Results}\label{statements}

Generally speaking, a \textit{Fr\'echet space} is a complete locally convex topological vector space whose topology can be generated by a translation-invariant metric. Then taking Minkowski functions of a countable local base of convex symmetric neighbourhoods of zero, we obtain countable family of seminorms, which also define the topology, see for details \cite[pp.110-114]{fabetal}.

In the context of Nash-Moser-Ekeland theory, however, a set of seminorms is given in advance and the estimates are \textit{in terms} of the given seminorms. Therefore, when the seminorms are fixed, we simply say that \textit{Fr\'echet space} $(X,\|\cdot\|_n)$ is a linear space $X$ with a collection
of seminorms $\|\cdot\|_{n}$, $n=0,\ldots,\infty$, which is separating, that is,
$\|x\|_{n}=0$, $\forall n,$ if and only if $x=0$; and, moreover, equipped with  the metric
\begin{equation}\label{eq:rho}
\rho_X(x,y):=\max_{n\ge 0}\frac{2^{-n}\|x-y\|_{n}}{1+\|x-y\|_{n}}
\end{equation}
$(X,\rho_X)$ is complete metric space.

Denote (as somewhat standard)
$$
  \mathbb{R}_+^\infty := \{(s_{i})_{i=0}^{\infty}:\ s_{i}\ge0\},
$$
that is, $\mathbb{R}_+^\infty$ is the cone of all positive sequences indexed from $0$ on.

For $\mathbf{s}\in \mathbb{R}_+^\infty$, $\mathbf{s}=(s_{i})_{i=0}^{\infty}$, define
\begin{equation}
	\label{eq:def:pi}
	\Pi_{\mathbf{s}}(X):=\{x\in X:\ \|x\|_{n}\le s_{n},\ \forall n\ge 0\}.
\end{equation}

Perhaps, the most used example of Fr\'echet space is $X = C^\infty (\Omega)$, where $\Omega$ is a compact domain in $\mathbb{R}^n$. Such spaces  are '\emph{tame}' in many aspects. The one we focus on, see Corollary~\ref{cor:compact}, is that $\Pi_{\mathbf{s}}$ are compact for all $\mathbf{s}\in \mathbb{R}_+^\infty$. For a proof see Proposition~\ref{pro:ci-comp}.

\begin{dfn}
  Let $(M_i,\rho_i)$, $i=1,2$, be linear metric spaces,
  $$
    F:M_1\rightrightarrows M_2
  $$
  be a multi-valued map, and $A\subset M_1$ be a non-empty set. For $(x,y)\in \grf F:=\{(x,y):\ y\in F(x)\}$ define $D'F (x,y)(A)\subset M_2$ by
  $$
    v\in D'F (x,y)(A)\iff \exists t_n\downarrow 0:\ \lim_{n\to \infty} \frac{\inf \rho_2 (F(x+t_nA),y+t_nv)}{t_n} = 0.
  $$
\end{dfn}
The above notion -- which essentially is defined in \cite{frankovska} -- extends the so called contingent, or graphical, derivative, see  \cite[pp.163, 202]{Ioffe_book}.

It is immediate from the definition that
%\begin{equation}\label{eq:mult:DF}
 \[ D'F (x,y) (tA) = t D'F (x,y) (A),\quad\forall t > 0.\]
%\end{equation}

For a point $y$ and a non-empty set $C$ in $M_2$, let us define $[y,y+C]:=\{ y+tc : c\in C,\ t\in [0,1]\}$.

We are ready to formulate our first cornerstone result.

\begin{thm}\label{thm:kornerstone}
  Let $(X,\|\cdot\|_n)$ and $(Y,|\cdot|_n)$ be Fr\'echet spaces and let
  $$
    F:X\rightrightarrows Y
  $$
  be a multi-valued map with closed graph.

  Let $U\subset X$ and $V\subset Y$ be open and such that $\gph F\cap (U\times V)\neq \emptyset$.

  Assume that for some $\mathbf{s}\in \mathbb{R}_+^\infty$ and some non-empty set  $C\subset Y$ it holds that
 \[
    D'F(x,y)(\Pi_{\mathbf{s}}(X))\supset  C,\quad \forall (x,y)\in \left( U\times V\right) \cap \grf F.
  \]
  Then,
  \[    \cl F(x+\Pi_{\mathbf{s}}(X)) \supset y + C\]
  for all $(x,y)\in\grf F$ such that $x+\Pi_{\mathbf{s}}(X)\subset U$, and $[y,y+C]\subset V$.
\end{thm}

It easily follows that, if $\Pi_{\mathbf{s}}(X)$ is compact, then the conclusion of Theorem~\ref{thm:kornerstone} holds without closure as we prove in the following corollary.

\begin{cor}\label{cor:compact}
  Let the assumptions of Theorem~$\ref{thm:kornerstone}$ hold and let, moreover, $\Pi_{\mathbf{s}}(X)$ be compact.

 Then,
  \[
 {F(x+\Pi_{\mathbf{s}}(X))}\supset y + C
 \]
 for  all $(x,y)\in\grf F$ such that $x+\Pi_{\s}(X)\subset  U$, and    $[y,y+C]\subset V$.
  \end{cor}

We use closed balls, so $B_X$ is the \textit{closed} unit ball of the Banach space $(X, \|\cdot\|)$, that is, $B_X:= \{x\in X:\ \|x\|\le 1\}$, while for a metric space $(M,\rho)$
$$
	B_\rho(x;r):= \{y\in M:\ \rho(x,y)\le r\}.
$$

To motivate the following definitions, let us recall Definition~2.35 from  \cite{Ioffe_book}: if $(M_i,\rho_i)$, $i=1,2$ are metric spaces, then the multi-valued map  $F:M_1\rightrightarrows M_2$
 is said to be \textit{restrictedly} $\gamma$-\textit{open at linear rate on} $(U,V)$ for sets $U\subset M_1$,  $V\subset M_2$, and extended real-valued function $\gamma$ on $M_1$ assuming positive values (possibly infinite) if there is an $r>0$ such that
\[
F(B_{\rho_1}(x,t))\supset B_{\rho_2}(y,rt)\cap V \]
whenever $(x,y)\in \grf F$, $x\in U$, $y\in V$ and $t<\gamma(x)$. Restricted  $\gamma$-openness at linear rate of $F$ is equivalent to its restricted  $\gamma$-\textit{metric regularity}  (see \cite{Ioffe_book}).

\begin{dfn}\label{def:pi-reg}
Let  $(X,\|\cdot\|_n)$ and $(Y,|\cdot|_n)$ be Fr\'echet spaces and let $U\subset X$ and $V\subset Y$ be nonempty and open.

The multi-valued map $F:X\rightrightarrows Y$ is said to be $\Pi$-\emph{surjective}
on $(U,V)$ if there is a constant $\kappa  > 0$ such that whenever $(x,y) \in (U\times V)\cap\grf F$ and $\s \in\mathbb{R}_+^\infty$ are such that $x+  \Pi_\s(X)\subset U$, it holds that
$$F(x+  \Pi_\s(X)) \supset \{y+ \kappa\Pi_\s(Y)\}\cap V.$$
\end{dfn}

We will need also the following straightforward weakening of the above notion.

\begin{dfn}\label{def:w-pi-reg}
Let  $(X,\|\cdot\|_n)$ and $(Y,|\cdot|_n)$ be Fr\'echet spaces and let $U\subset X$ and $V\subset Y$ be nonempty and open.

The multi-valued map $F:X\rightrightarrows Y$ is said to be \emph{weakly} $\Pi$-\emph{surjective}
on $(U,V)$ if there is a constant  $\kappa  > 0$ such that whenever $(x,y) \in (U\times V)\cap\grf F$ and $\s \in\mathbb{R}_+^\infty$ are such that $x+ \Pi_\s(X)\subset U$, it holds that
$$\cl {F(x+  \Pi_\s(X))} \supset\{y+ \kappa\Pi_\s(Y)\}\cap V.$$
\end{dfn}

It is not clear whether weak $\Pi$-surjectivity is equivalent to $\Pi$-surjectivity. However, we have the following implication.

\begin{thm}\label{thm:w-reg-m-reg}
  Let  $(X,\|\cdot\|_n)$ and $(Y,|\cdot|_n)$ be Fr\'echet spaces and let $U\subset X$ and $V\subset Y$ be nonempty and open.

  If the multi-valued map $F:X\rightrightarrows Y$ with closed graph is weakly $\Pi$-surjective on $(U, V)$ then there is $\theta >0$ such that
  \[
\begin{array}{c}
F(B_{\rho_X}(x;r))\supset B_{\rho_Y}(y;\theta r)\cap V,\\[10pt]
\forall (x,y)\in\grf F\cap (U\times V),\ \forall r:\ 0<r< m_U(x),
 \end{array}
\]
where
$$
  m_U(x) := \mathrm{dist}(x,X\setminus U).
$$
That is, $F$ is restrictedly $\gamma$-open at liner rate on $(U,V)$ for $\gamma (x)=m_U(x)$.
\end{thm}

In this context, Theorem~\ref{thm:kornerstone} easily yields another of our main results, namely

\begin{thm}\label{thm:main-pi-surj}
  Let $(X,\|\cdot\|_n)$ and $(Y,|\cdot|_n)$ be Fr\'echet spaces and let
  $$
    F:X\rightrightarrows Y
  $$
  be a multi-valued map with closed graph.

  Let $U\subset X$ be open and let  $V\subset Y$ be open and convex. Assume that for some $\kappa>0$
  \[
    D'F(x,y)(\Pi_{\mathbf{s}}(X))\supset \kappa \Pi_{\mathbf{s}}(Y),\quad \forall (x,y)\in \left( U\times V\right) \cap \grf F,\ \forall \mathbf{s}\in \mathbb{R}_+^\infty.
  \]
  Then, $F$ is weakly $\Pi$-surjective on $(U, V)$ with constant $\kappa$.

  If, moreover, $\Pi_{\mathbf{s}}(X)$ are compact for all $\s\in \mathbb{R}_+^\infty$, then $F$ is $\Pi$-surjective on $(U, V)$ with constant $\kappa$.
\end{thm}

Finally, we present a corollary for the case of function, that is, single-valued map.

	We call the space $X$ \textit{compactly graded} if $X=\cap_0^\infty X_n$, where $(X_n,\|\cdot\|_n)$ are nested Banach spaces: $X_{n+1}\subset X_n$ and the identity operator from $X_{n+1}$ into $X_n$ is compact. We mostly have in mind $C^\infty(\Omega)$, where $\Omega\subset \R^n$ is compact domain.

The function $f:M_1\to M_2$, where $(M_i,\rho_i)$ are metric spaces, is called \textit{strongly G\^ateaux differentiable} at $x\in M_1$ if there exists a bounded linear operator $f'(x):M_1\to M_2$ such that
$$
\lim_{t\downarrow 0} \frac{\rho_2(f(x+th)-f(x),t f'(x)(h))}{t} = 0,\quad \forall h\in M_1.
$$
Any strongly G\^ateaux differentiable function $f:M_1\to M_2$ is G\^ateaux differentiable, i.e.
$$
\lim_{t\downarrow 0} \rho_2\left(\frac{f(x+th)-f(x)}{t}, f'(x)(h)\right) = 0,\quad \forall h\in M_1.
$$

\begin{thm}\label{thm:ci}
	Let $X=\cap_0^\infty(X_n,\|\cdot\|_n)$ and $Y=\cap_0^\infty (Y_n,|\cdot|_n)$ be compactly graded spaces. Let $f:X\to Y$ be continuous, strongly G\^ateaux differentiable  and such that $f(0)=0$.
	
	Assume that there are $c>0$ and $d\in \{0\}\cup\mathbb{N}$, such that for each $x\in X$ and $v\in Y$
	$$
	\exists u\in X:\ f'(x)u=v\mbox{ and }
	\|u\|_n\le c|v|_{n+d},\quad\forall n\ge 0.
	$$	
	Then for each  $y\in Y$ there is $x\in X$ such that
	$$
	f(x)=y\mbox{ and }
	\|x\|_n\le c|y|_{n+d},\quad\forall n\ge 0.
	$$		
\end{thm}

Comparing this statement to  \cite[Theorem~1]{ekeland}, we note that our assumptions on $X$ and the differentiability of $f$ are more restrictive. Nonetheless, they  include the most important cases of infinitely smooth functions on compacts. On the other hand, we do not bound the norms of the derivatives and do not require the existence of left inverse of these derivatives.

The most significant difference is that in the conclusion we have estimates for \textit{all} norms simultaneously.

Note that by a different method the above statement is proved for merely G\^ateaux differentiable function in \cite{iz-dban}.

\section{Preliminaries}\label{sec:prelim}

We recall for future use

\noindent \textbf{Ekeland variational principle} (e.g. Phelps~\cite[p.45]{Phelps}).
Let $f$ be a proper lower semi\-con\-ti\-nu\-ous function from a complete metric space $(M,\rho)$ into\linebreak 
$\mathbb{R}\cup\{+\infty\}$. Let $f$ be bounded below and $\eps > 0$, $\la >0$ and $\hat y$ be such that $f(\hat y)<+\infty$ and
$f(\hat y) \le \inf f + \varepsilon \la$. Then  there is $\hat x\in \mathrm{dom}\, f$
such that:
\begin{itemize}
	\item[\rm (\i)]
	$ \lambda \rho( \hat x, \hat y) \le f(\hat y) - f(\hat x)$,
	\item[\rm (\i\i)]
	$\rho( \hat x, \hat y)\le \varepsilon$, and
	\item[\rm (\i\i\i)]
	$ \lambda \rho( x,\hat x) + f(x) \ge f(\hat x)$ for all $x\in M$.
\end{itemize}

The next statement is a basic exercise in Functional Analysis.

\begin{lem}
	\label{lem:re-metric}
	Let $(X,\rho)$ be a locally convex space with shift-invariant metric $\rho$. Then, for each fixed $\bar x\in X\setminus \{0\}$ there is an equivalent shift-invariant metric $\bar\rho$ such that
	$$
		\bar\rho(0,t\bar x) = |t|,\quad\forall t\in\mathbb{R}.
	$$
\end{lem}

\begin{proof}
	Let $\bar x\in X\setminus \{0\}$.
	
	Since $X$ is locally convex, by Hahn-Banach Theorem (see e.g. \cite{holms}) there is a continuous linear functional $p$ on $X$ such that $p(\bar{x})=1$. Define
	$$
		\bar\rho(0,x):= |p(x)|+\rho(0,x-p(x)\bar x),
	$$
	and $\bar\rho(x,y):=\bar\rho(0,x-y)$. Then
	 $\bar\rho$ is shift-invariant by construction and it is easy to check that $\bar\rho(0,t\bar x)=|p(t\bar x)|=|t||p(\bar x)|=|t|$.
	
	Since $p$ is continuous, it is clear that if $\rho(0,x_n)\to 0$ then $\bar\rho(0,x_n)\to 0$.
	
	Conversely,  if $\bar\rho(0,x_n)\to 0$ then from the definition of $\bar\rho$ it is clear that $p(x_n)\to 0$ and $\rho(0,x_n-p(x_n)\bar x)\to 0$.

 Since  $\rho(0,x_n)\le \rho(0,p(x_n)\bar x)+\rho(0,x_n-p(x_n)\bar x)$ and $\rho(0,p(x_n)\bar x)\to 0$ (because the scalar multiplication is continuous), $\rho(0,x_n)\to 0$. 
\end{proof}

\begin{dfn}
  For $\mathbf{s}\in \mathbb{R}_+^\infty$ set
%  \begin{equation}\label{eq:def:supp:s}
\[    \supp \mathbf{s} : = \{n\ge 0:\ s_n > 0\},\]
%  \end{equation}
%  \begin{equation}\label{eq:def:mod:s}
 \[   |\mathbf{s}| : = \max_{n\ge 0} \frac{2^{-n}s_{n}}{1+s_{n}}.\]
%  \end{equation}
  For a Fr\'echet space $(X,\|\cdot\|_n)$ and a given $\mathbf{s}\in \mathbb{R}_+^\infty$ also define (recall (\ref{eq:def:pi}))
%  \begin{equation}\label{eq:def:Pi:s}
\[  \Pi_{\mathbf{s}}(X):=\{x\in X:\ \|x\|_{n}\le s_{n},\ \forall n\ge 0\},\]
%\end{equation}
and
  \begin{equation}\label{eq:def:X:s}
    X_{\mathbf{s}}:=\bigcup_{t\ge 0}t\Pi_{\mathbf{s}}(X).
  \end{equation}
   \end{dfn}
We will now discuss some properties of the structure thus introduced to be used in the sequel.

What is obvious is that, if $x\in X$ and $\mathbf{s}=(\|x\|_n)_{i=0}^{\infty}$, then
%\begin{equation}\label{eq:mod:x}
\[  |\mathbf{s}| = \rho_X(0,x),\]
%\end{equation}
where $\rho_X$ is defined by \eqref{eq:rho}.

The following is one of the key relations in the approach we present.
\begin{lem}\label{lem:key}
  If $(X,\|\cdot\|_n)$ is Fr\'echet space and $\rho_X$ is defined as in \eqref{eq:rho}, then
  $$
    c\Pi_{\mathbf{s}}(X)\subset B_{\rho_X}(0; c |\mathbf{s}|)
  $$
  for any $\mathbf{s}\in \mathbb{R}_+^\infty$ and any $c\ge 1$.
\end{lem}

\begin{proof}
  Fix arbitrary  $\mathbf{s}\in \mathbb{R}_+^\infty$ and $c\ge1$.

  Set
  $$
    g(t):=\frac{t}{1+t},\quad t\ge 0,
  $$
  so $\rho_X(0,x) = \max_{n\ge 0}2^{-n}g(\|x\|_n)$ and $|\mathbf{s}| = \max_{n\ge 0}2^{-n}g(s_n)$. 
  From the concavity of $g$ and $g(0)=0$ it follows that $g(t) = g(c^{-1}(ct))\ge c^{-1}g(ct)$, since $c^{-1}\in \,]0,1[\,$. Therefore,
  \begin{equation}\label{eq:gct}
    g(ct)\le cg(t),\quad\forall t\ge0,
  \end{equation}
  which, of course, can be checked just as well directly.

  If $x\in c\Pi_{\mathbf{s}}(X)$ then $\|x\|_n\le cs_n$ and then $g(\|x\|_n)\le g(cs_n)$, because $g$ is increasing. From \eqref{eq:gct} it follows that $g(\|x\|_n)\le cg(s_n)$ and, therefore, $\rho_X(0,x)= \max_{n\ge 0}2^{-n}g(\|x\|_n)\le \max_{n\ge 0}2^{-n}cg(s_n)\le c|\mathbf{s}|$. 
\end{proof}

It is easy to check that $\Pi_{\mathbf{s}}(X)$ is closed in $X$, because of
$$
  \lim_{k\to\infty}\rho_X(x_k,x)=0\iff \lim_{k\to\infty}\|x_k-x\|_n = 0,\ \forall n\ge 0,
$$
so the inequalities $\|\cdot\|_n\le s_n$ are preserved by $\rho_X$-convergence.

We will need the following estimate.

\begin{lem}\label{lem:diam:Pi:estimate}
Let $(X,\|\cdot\|_n)$ be a Fr\'echet space and $\rho_X$ be defined as in $\eqref{eq:rho}$.   Let $\mathbf{s}\in \mathbb{R}_+^\infty$. Then,
  $$
    \lim_{t\downarrow 0}\sup\rho_X(0,t\Pi_{\mathbf{s}}(X)) = 0.
  $$
\end{lem}

\begin{proof}
  If $t>0$, then $x\in t \Pi_{\mathbf{s}}(X)\iff \|x\|_n\le ts_n$, $\forall n\ge 0$.

  Let $\varepsilon>0$. Fix $N>1/\varepsilon$. Since $2^{-N}<\varepsilon$, we have for all $x\in t \Pi_{\mathbf{s}}(X)$
  $$
    \rho_X(0,x) \le \max_{0\le n\le N}\frac{2^{-n}ts_n}{1+ts_n} + \varepsilon<t\max_{0\le n\le N}s_n + \varepsilon < 2\varepsilon
  $$
  for $t$ small enough. 
\end{proof}

\begin{lem}\label{lem:X_s:Banach}
Let $(X,\|\cdot\|_n)$ be a Fr\'echet space, let $\mathbf{s}\in \mathbb{R}_+^\infty$ be such that\linebreak  
$\supp \mathbf{s}\ne\emptyset$ and let $X_\mathbf{s}$ be defined by $(\ref{eq:def:X:s})$.
  For $x\in X_\mathbf{s}$   define
%  \begin{equation}\label{eq:def:norm:s}
 \[   \|x\|_{\mathbf{s}}:=\sup \left\{\frac{\|x\|_{n}}{s_{n}}:\ n\in\supp \s\right\}.\]
%  \end{equation}
  Then, $\Pi_{\mathbf{s}}(X)$ is the unit ball of the norm $\|\cdot \|_{\mathbf{s}}$ and $(X_{\mathbf{s}},\|\cdot\|_{\mathbf{s}})$ is a Banach space.
\end{lem}

\begin{proof}
  Clearly, $\|\cdot\|_{\mathbf{s}}$ is a norm on $X_{\mathbf{s}}$ and that its unit ball  is $\Pi_{\mathbf{s}}(X)$.

To prove that $(X_{\mathbf{s}},\|\cdot\|_{\mathbf{s}})$ is a Banach space,  let us take a $\|\cdot\|_{\mathbf{s}}$-Cauchy
  sequence  $(x_{k})_{k=1}^{\infty}$   in $X_{\mathbf{s}}$, that is,
  \[
    \lim_{k,m\to\infty}\|x_{k}-x_{m}\|_{\mathbf{s}}=0.
  \]
  Then, $(x_{k})_{k=1}^{\infty}$ is $\|\cdot\|_{\mathbf{s}}$-bounded and by multiplying it by suitable positive constant we may assume without loosing generality that $(x_{k})_{k=1}^{\infty}\subset \Pi_{\mathbf{s}}(X)$.

  For any $t>0$ and for all $k$ and $m$ large enough $x_{k}-x_{m}\in t \Pi_{\mathbf{s}}(X)$. From Lemma~\ref{lem:diam:Pi:estimate} it follows that $(x_{k})_{k=1}^{\infty}$ is also $\rho_X$-Cauchy and, therefore, convergent in $(X,\rho_X)$. Since $\Pi_{\mathbf{s}}(X)$ is $\rho_X$-closed in $X$, there is $\bar x\in \Pi_{\mathbf{s}}(X)$ such that
  $$
    \lim_{k\to\infty}\rho_X(\bar x,x_k) = 0.
  $$
  Fix $\varepsilon > 0$ and $N\in \mathbb{N}$ such that $\|x_m-x_k\|_{\mathbf{s}}<\varepsilon$ for all $k>m>N$. Fix some $m>N$. The sequence $(x_m-x_k)_{k=m+1}^\infty$ is in the $\rho_X$-closed set $\varepsilon \Pi_{\mathbf{s}}(X)$ and $\rho_X$-converges to $x_m-\bar x$. This means that $x_m-\bar x\in \varepsilon \Pi_{\mathbf{s}}(X)$, or, equivalently, $\|x_m-\bar x\|_{\mathbf{s}}\le \varepsilon$ for arbitrary $m>N$. That is, $\|x_k-\bar x\|_{\mathbf{s}}\to 0$. 
\end{proof}

The next result is straightforward.

\begin{prop}\label{pro:ci-comp}
	Let $X$ be a compactly graded space. Let $X=\cap_0^\infty X_n$, where $(X_n,\|\cdot\|_n)$ are nested Banach spaces with compact embedding  $X_{n+1}\hookrightarrow X_n$. Then $(X,\|\cdot\|_n)$ is a Fr\'echet space and, moreover, $\Pi_{\mathbf{s}}(X)$ is compact for each $\mathbf{s}\in\mathbb{R}_+^\infty$.
\end{prop}

\begin{proof}
	If a sequence is Cauchy in $X$ then it will be Cauchy and, therefore, convergent in each of $X_n$'s, so also convergent in $X$. Therefore, $X$ is complete.
	
	Clearly,
	$$
		\Pi_{\mathbf{s}} (X)= \bigcap_{n\ge 0} s_nB_{X_n}
	$$
	is closed. We will show that it has finite $\varepsilon$-net for each $\varepsilon>0$.
	
	To this end fix $\varepsilon>0$ and let $k$ be so large that for
	$$
		 \rho_k (x,y):= \max_{0\le n\le k}\frac{2^{-n}\|x-y\|_{n}}{1+\|x-y\|_{n}},\quad\forall x,y\in X_k,
	$$
	and $\rho_X$ defined by \eqref{eq:rho} it is fulfilled that
	\begin{equation}
		\label{eq:rho-rho-hat}
		|\rho_k(x,y)-\rho_X(x,y)|<\varepsilon,\quad\forall x,y\in X.
	\end{equation}
	Since $\|\cdot\|_k$ is stronger than $\|\cdot\|_1,\ldots,\|\cdot\|_{k-1}$, the metric $ \rho_k (x,y)$ defines the same topology on $X_k$. Since by assumption $s_{k+1}B_{X_{k+1}}$ is a compact subset of $X_k$, there is a finite $\varepsilon$-net to  $s_{k+1}B_{X_{k+1}}$ in $(X_k,\rho_k)$. That is, there are $a_1',\ldots,a_m'\in X_k$ such that
	$$
		s_{k+1}B_{X_{k+1}}\subset\bigcup_{i=1}^m B_{(X_k,\rho_k)}(a_i',\varepsilon).
	$$
	Since $\Pi_{\mathbf{s}}(X)\subset s_{k+1}B_{X_{k+1}}$, $a_1',\ldots,a_m'$ is an  $\varepsilon$-net to $\Pi_{\mathbf{s}}(X)$ in $(X_k,\rho_k)$. Let
	$$
		I:=\{ i\in \{1,\ldots,m\}:\ B_{(X_k,\rho_k)}(a_i',\varepsilon)\cap \Pi_{\mathbf{s}}(X)\neq\emptyset\}.
	$$

Obviously, $(a_i')_{i\in I}$ is a $\varepsilon$-net to $\Pi_{\mathbf{s}}(X)$ in $(X_k,\rho_k)$. For each $i\in I$ chose $a_i\in B_{(X_k,\rho_k)}(a_i',\varepsilon)\cap \Pi_{\mathbf{s}}(X)$. By triangle inequality $(a_i)_{i\in I}$ is a $2\varepsilon$-net to $\Pi_{\mathbf{s}}(X)$ in $(X_k,\rho_k)$. From \eqref{eq:rho-rho-hat} it follows that $(a_i)_{i\in I}$ is a $3\varepsilon$-net to $\Pi_{\mathbf{s}}(X)$ in $(X,\rho_X)$. 
\end{proof}

The following can be derived from \cite[Theorem 2.55]{Ioffe_book}, but we present a proof for reader's convenience.

\begin{lem}\label{lem:density}
  Let $(M_1,\rho_1)$ and $(M_2,\rho_2)$ be complete metric spaces. Let $F:M_1\rightrightarrows M_2$ be a multi-valued map with closed graph. Let $U\subset M_1$ and $V\subset M_2$ be open and such that $\grf F\cap (U\times V)\neq \emptyset$.

  Let for  some $\alpha>0$ it hold that
  $$
    \cl {F(B_{\rho_1}(x;r))}\supset B_{\rho_2}(y;\alpha r)\cap V,
  $$ for all $(x,y)\in\grf F\cap (U\times V)$ and $r<m_U(x):=\mathrm{dist}(x,M_1\setminus U)$.
  Then for  $\beta \in ]0,\alpha[$ it holds that
%\begin{equation}\label{eq:mrk-2}
\[\begin{array}{l}
F(B_{\rho_1}(x;r))\supset B_{\rho_2}(y;\beta r)\cap V,\\[10pt]
\forall (x,y)\in\grf F\cap (U\times V),\ \forall r<m_U(x),
 \end{array}\]
%\end{equation}
that is, $F$ is restrictedly $\gamma$-open at linear rate  on $(U,V)$ for $\gamma(x)=m_U(x)$.
\end{lem}

\begin{proof}
  Fix arbitrary $(x_0,y_0)\in\grf F \cap (U\times V)$.

  Let $r>0$ be such that $r<m_U(x_0)$,
   so $B_{\rho _1}(x_0,r)\subset U$.
  It is clear that
  $$
    M := (\overline U \times \overline V)\cap \grf F
  $$
  is complete in the product metric
  $$
    \rho ((x,y),(u,v)) := \max\{ \rho_1(x,u),\alpha^{-1}\rho_2(y,v)\}.
  $$
 Take arbitrary $\beta \in ]0,\alpha[$.

  Fix arbitrary $\xi\in B_{\rho_2}(y_0;\beta r)\cap V$. Define $g:M\to\mathbb{R}$ by
  $$
    g(x,y) := \rho_{2}(y,\xi).
  $$
  Then $g\ge0$ and $g(x_0,y_0)\le\beta r$. By Ekeland Variational Principle (with $\la=\beta$ and $\varepsilon=r$) there is $(x_1,y_1)\in M$ such that
  $$
    \rho((x_1,y_1),(x_0,y_0))\le r \Rightarrow\rho_1(x_1,x_0)\le r,\ \rho_2(y_1,y_0)\le \alpha r,
  $$
    \begin{equation}\label{eq:Ekeland}
    g(x,y)-g(x_1,y_1) \ge -\beta \rho((x,y),(x_1,y_1)),\quad \forall (x,y)\in M,
  \end{equation}
  and
  $$
  \beta \rho((x_0,y_0),(x_1,y_1))\le \beta r-g(x_1,y_1)= \beta r -\rho_2(y_1,\xi).
  $$

  Set $ p:=\rho_2(y_1,\xi)$. From the latter we have that $p\le \beta (r-\rho_1(x_1,x_0))$.

 Suppose that $p>0$. Take $r'$ such that $r<r'<m_U(x_0)$. Then
 for any $x\in B_{\rho_1}(x_1;\alpha^{-1}p+r'-r)$ we have that
 \begin{eqnarray*}
 \rho_1(x,x_0)&\le &\rho_1(x,x_1)+\rho_1(x_1,x_0) \le \alpha^{-1}p +r'-r+ \rho_1(x_1,x_0)\\
                &\le&\alpha^{-1}\beta (r-\rho_1(x_1,x_0))+ r'-r+\rho_1(x_1,x_0)\\
                &=&\beta \alpha^{-1}r+(1-\beta \alpha^{-1})\rho_1(x_1,x_0)+r'-r\le r',
  \end{eqnarray*}
using that $   \rho_1(x_1,x_0)\le r$.

  So,  $B_{\rho_1}(x_1;\alpha^{-1}p+r'-r)\subset B_{\rho_1}(x_0;r')\subset U$ and hence $\alpha^{-1}p{+}r'{-}r{\le} m_U(x_1)$. Therefore, $\alpha^{-1}p< m_U(x_1)$.

  By  assumption, $\cl{F(B_{\rho_1}(x_1;\alpha^{-1}p))}\supset B_{\rho_2}(y_1;p)\cap V$.

  Since $\xi\in B_{\rho_2}(y_1;p)\cap V$, the latter implies $\xi\in \cl{F(B_{\rho_1}(x_1;\alpha^{-1}p))}$.

  This means that there is a sequence $(u_k,v_k)^\infty_1\in M$ such that $\rho_1(x_1,u_k)\le \alpha^{-1}p$ and $\rho_2(v_k,\xi)\to 0$. Thus we can substitute $(x,y)=(u_k,v_k)$ in \eqref{eq:Ekeland} to get
  $$
    \rho_2(v_k,\xi)-\rho_2(y_1,\xi) \ge -\beta\max\{ \rho_1(u_k,x_1),\alpha^{-1}\rho_2(v_k,y_1)\}.
     $$
     
  Taking into account that $\rho_2(v_k,\xi)\to 0$, $\rho_2(y_1,\xi)=p$ and\linebreak 
  $\rho_2(v_k,y_1)\to \rho_2(\xi,y_1)=p$, we get $-p \ge - \beta \alpha^{-1}p$. But $\beta <\alpha$, hence $p=0$ and we get a contradiction.

Therefore $p=0$, that is, $y_1=\xi$, so $\xi\in F(x_1)\subset F(B_{\rho_1}(x_0;r))$.

  Since $\xi\in B_{\rho_2}(y_0;\beta r)\cap V$ was arbitrary, we conclude that
  $$
    F(B_{\rho_1} (x_0;r))\supset B_{\rho_2}(y_0;\beta r)\cap V
  $$
and the proof is completed.  
 \end{proof}

\section{LOEV Principle}\label{sec:loev}

 In \cite{LOEV} we have established the so called \emph{Long Orbit or Empty  Value (LOEV) Principle}   and we have used it there for getting surjectivity results in Banach spaces.

Let $\left(M,\rho\right)$ be a complete metric space.

  Let $S:M\rightrightarrows M$ be a multi-valued map. We say that $S$
  \textbf{satisfies the condition} $\left(\ast\right)$ if $x\notin S(x)$, $\forall x\in M$,
  and whenever $y\in S(x)$ and $\lim_{n}x_{n}=x$, there are infinitely
  many $x_{n}$'s such that
  $
  y\in S(x_{n}).
  $

  The following slight modification of LOEV Principle  is proved here by a minor alteration of the original proof in \cite{LOEV}. We give this proof only for the sake of completeness.
  \begin{thm}\label{thm:loev-set-version}
Let $\left(M,\rho\right)$ be a complete metric space and let $S:M\rightrightarrows M$ satisfy~$\left(\ast\right)$.
    Let $M'\subset M$ be such that
    $$
      S(x)\ne\emptyset,\quad\forall x\in M'.
    $$
    If $x_0\in M$ then one of the conditions $(a)$, $(b1)$ and $(b2)$  below is true $((a)$ corresponds to infinite length orbit while $(b1)$ and $(b2)$ correspond to finite length orbit$)$.

    $(a)$ There are $x_{i}\in M$, $i=1,2,\ldots$, such that
    \[
    x_{i+1}\in S(x_{i}),\ i=0,1,\ldots;
    \quad
    \sum_{i=0}^{\infty}\rho(x_{i},x_{i+1})=\infty;
    \]

    $(b1)$ There are $x_{i}\in M$, $i=0,1,2,\ldots,n$, such that
    \[
    x_{i+1}\in S(x_{i}),\  i=0,1,\ldots,n-1;
    \quad
    x_{n}\not\in M';
    \]

    $(b2)$ There are $x_{i}\in M$, $i=1,2,\ldots$, such that
    \[
    x_{i+1}\in S(x_{i}),\ i=0,1,\ldots;
    \quad
    \sum_{i=0}^{\infty}\rho(x_{i},x_{i+1})<\infty,\quad \mbox {and}\quad x_i\to \ol x\not \in M'.
    \]
  \end{thm}

  \begin{proof}
    We can construct finite or infinite orbit $( x_{i} )_{i\ge 0}\subset M$ by the following procedure.

  If $x_0, x_1,\ldots,x_i$ are already chosen, then

  either $S(x_i) = \emptyset$, which means that $x_i\not\in M'$: that is, (b1) is fulfilled;

  or $s_{i}:=\min\left\{ 1,\sup\left\{ \rho(x_{i},y):\,y\in S(x_{i})\right\} \right\} > 0$.
  Take $x_{i+1} \in S(x_i)$ such that
  \begin{equation}
  \rho(x_{i},x_{i+1})>\frac{s_{i}}{2}.\label{eq:half-sup}
  \end{equation}

  If we end up with infinite orbit then either (a) is fulfilled, or
  \[
  \sum_{i=0}^{\infty}\rho(x_{i},x_{i+1})<\infty\Rightarrow\lim_{i\to\infty}x_{i}=:\bar{x}.
  \]
  In the latter case, $\lim_{i}\rho(x_{i},x_{i+1})=0$ and from (\ref{eq:half-sup})
  it follows that $s_{i}\to 0.$

Assume that $S(\bar{x}) \neq \emptyset$.

  Take $\bar y\in S(\bar{x})$.
  By $\left(\ast\right)$ we have $\rho(\bar y,\bar{x})>0$
and $\bar y\in S(x_{i})$
  for infinitely many $i$'s. By the definition of $s_i$ we have that $s_{i}\ge \rho(\bar y,x_i)$ for infinitely
  many $i$'s. Passing to limit over the latter subsequence we get $0\ge \rho(\bar y,\bar x) > 0$.
  Contradiction. Hence $S(\ol x)=\emptyset$ and then $\ol x\not \in M'$. Thus, (b2) is fulfilled. 
\end{proof}

We apply LOEV principle to prove the following proposition that we will need in the sequel.

\begin{prop}\label{prop:4}
Let $(X,\|\cdot\|)$ be a Banach space and let $(Y,\rho)$
be a Fr\'echet space.
Let $G:X\rightrightarrows Y$ be a multi-valued map with closed graph and\linebreak 
$(x_0,y_0)\in \grf G$.  If for some $ \bar y\in Y$ and some open set $V\supset [y_0,y_0+\bar y]$
 there exist $\sigma >0$ and an open set $W\supset x_0+\sigma B_X$   such that
\[
\bar y\in D'G(x,y)(\sigma B_{X}),\quad\forall (x,y)\in \grf G \cap \left(W\times V\right),
\]
then
\[
y_0+\bar y\in\cl {G(x_0+\sigma B_{X})}.
\]
\end{prop}

\begin{proof}
Obviously, we may assume that $\bar y\ne 0$. From Lemma~\ref{lem:re-metric} it follows that there is no loss of generality if we assume that the metric $\rho$ on $Y$ is such that
\begin{equation}
	\label{eq:t-bar-y}
	\rho(0,t\bar y) = |t|,\quad\forall t\in\mathbb{R}.
\end{equation}

Fix arbitrary $\varepsilon_1 \in \, ]0,1[\,$.

Since $[y_0,y_0+\bar y]$ is compact, there is $\varepsilon\in\,]0,\varepsilon_1/2[$  such that
\begin{equation}\label{eq:def:eps}
  [y_0,y_0+\bar y]+B_\rho (0;\varepsilon)\subset V.
\end{equation}

Fix one such $\varepsilon$ and take $\mu > \sigma$ such that $\mu (1-\varepsilon)<\sigma$.

Let
$$
  M := \grf G \cap \left(\overline{W}\times\overline V\right).
$$
Then $M$ is complete in the product metric
$$
  \hat \rho((x,y),(u,v)):= \|x-u\|+\rho(y,v).
$$
Define $S:M\rightrightarrows M$ by
\[
S(x,y)=\left\{ (u,v)\in M\setminus\left\{ (x,y)\right\} :\,\exists t\in(\mu^{-1}\|u-x\|,\varepsilon):\ \rho(v-y,t\bar y)<\varepsilon t\right\} .
\]

We claim that $S(x,y)\neq\emptyset$ for all $(x,y)\in M'$, where
$$
  M' := \grf G \cap \left( W\times V\right).
$$

Indeed, fix $(x,y)\in M'$. Since by assumption
$\bar y\in D'G(x,y)(\sigma B_{X})$, by definition there are $h_n\in X$ such that $\|h_n\|\le\sigma $, $z_n\in G(x+s_nh_n)$ and $s_n\downarrow 0$ such that
\[
\frac{\rho(z_n, y+s_n\ol y)}{s_n}\to 0,\mbox{ as }n\to\infty.
\]
So, for $n$ large enough, $s_n\in\,]0,\varepsilon[$, $x+s_nh_n \in W$, and   $\rho(z_n-y,s_n\bar y)<\varepsilon s_n$.
Also, $s_n\sigma \ge\|s_nh_n\|=\|(x+s_nh_n)-x\|$. Therefore, $s_n\mu>\|(x+s_nh_n)-x||$, so $(x+s_nh_n,z_n)\in S(x,y)$ for all large enough $n$.

By definition $(x,y)\notin S(x,y)$. The other requirement of $(\ast)$ follows by continuity:
if $(u,v)\in S(x,y)$ and $(u_{n},v_n)\to (x,y)$ then for the $t$ corresponding
to $(u,v)$ in the definition of $S(x,y)$, we will have $t\in(\mu^{-1}\|u-u_{n}\|,\varepsilon)$ and
$\rho(v-v_n-t\bar y,0)<\varepsilon t$
for $n$ large enough.

Thus $S$ satisfies $(\ast)$. Therefore, there is a $S$-orbit starting at $(x_0,y_0)$ which satisfies (a), (b1), or (b2) from  Theorem~\ref{thm:loev-set-version}. Denote this orbit by
$$
  \{(x_i,y_i)\}_{i\in I},\quad\mbox{where } (x_{i+1},y_{i+1})\in S(x_i,y_i).
$$
For each $n\in \mathbb N$ such that $\{0,1,\ldots,n+1\}\subset I$ set
$$
  p_n := \sum_{i=0}^nt_i
$$
and note that from the definition of $S$ we immediately get
\begin{equation}\label{eq:s:t}
 \sum_{i=0}^{n}\|x_{i+1}-x_{i}\|\le \mu\sum_{i=0}^{n}t_i=\mu p_n.
\end{equation}
On the other hand,
\[
\rho\left(y_{n+1}-y_0,p_n\bar y\right)=\rho\left(\sum_{i=0}^{n}(y_{i+1}-y_{i}),\left(\sum_{i=0}^{n}t_{i}\right)\bar y\right).
\]
From shift-invariance of the metric $\rho$ it follows that
\begin{eqnarray*}
\rho\left(y_{n+1}-y_0,p_n\bar y\right) & = & \rho\left(0,\sum_{i=0}^{n}(y_{i+1}-y_{i}-t_{i}\bar y)\right)\\
 & \le & \sum_{i=0}^{n}\rho(0,y_{i+1}-y_{i}-t_{i}\bar y)\\
 & = & \sum_{i=0}^{n}\rho(y_{i+1}-y_{i},t_{i}\bar y)\\
 & \le & \varepsilon\left(\sum_{i=0}^{n}t_{i}\right)=\varepsilon p_n.
\end{eqnarray*}
So,
\begin{equation}\label{eq:y:smaller}
  \rho\left(y_{n+1}-y_0,p_n\bar y\right)\le \varepsilon p_n.
\end{equation}
If $p_n\le1-\varepsilon$ then from \eqref{eq:s:t} it follows that
\begin{equation}\label{dob}
\|x_{n+1}-x_0\|\le \sum_{i=0}^{n}\|x_{i+1}-x_{i}\|\le \mu \sum _{i=0}^n t_i=\mu p_n\le \mu(1-\varepsilon)<\sigma
 \end{equation}
and $x_{n+1}\in B_X^\circ(x_0;\sigma)$; while from \eqref{eq:y:smaller} it follows that
\begin{equation}\label{doby}
y_{n+1}\in B_\rho(y_0+p_n\bar y;\varepsilon p_n).
 \end{equation}

 Since $y_0+p_n\bar y\in [y_0,y_0+\bar y]$, because $p_n\in [0,1]$, from \eqref{eq:def:eps} it follows that $y_{n+1}\in V$. Thus we see that
$$
  p_n\le 1-\varepsilon\Rightarrow (x_{n+1},y_{n+1})\in M'.
$$

Assume that $p_n\le 1-\varepsilon$ for all $n$ such that $\{0,1,\ldots,n+1\}\subset I$. Then from the above implication it follows that the orbit is contained in $M'$. Recalling the way it was chosen from LOEV Principle, the orbit must then satisfy (a) or (b2) from Theorem~\ref{thm:loev-set-version}. In particular the orbit is infinite, thus $I=\{0\}\cup\mathbb{N}$. From \eqref{eq:s:t} it follows that $\sum_{i=0}^\infty \|x_{i+1}-x_i\|\le \mu(1-\varepsilon)<\infty$. On the other hand, $|\rho(y_{i+1},y_i)-\rho(0,t_i\bar y)|\le \rho(y_{i+1}-y_i,t_i\bar y)\le\varepsilon t_i$. Since
$\rho(0,t_i\bar y) = t_i$ (see  \eqref{eq:t-bar-y}), we get $\rho(y_{i+1},y_i)\le (1+\varepsilon)t_i$. Therefore, $\sum_{i=0}^\infty \rho(y_{i+1},y_i) \le  (1+\varepsilon)\sum_{i=0}^\infty t_i \le 1-\varepsilon^2$. That is, the orbit's length $\sum_{i=0}^\infty \hat\rho((x_i,y_i),(x_{i+1},y_{i+1}))$ is finite. Hence the orbit fulfills (b2) from Theorem~\ref{thm:loev-set-version}. Then $(x_i,y_i)\to (x',y')$, and
\begin{equation}\label{contr}
(x',y')\not\in M'.
 \end{equation}

 Since $p_n$ is monotone increasing and bounded, $p_n\to p'$ as $n\to \infty$. Moreover, $p'\in [0,1-\varepsilon]$ and $y_0+p'\ol y\in [y_0,y_0+\ol y]$.

For sufficiently large $n$, $\rho (y', y_{n+1})<\varepsilon^2/2$ and $\rho (y_0+p_n\ol y, y_0+p'\ol y )=\rho (0,(p'-p_n)\ol y)<\varepsilon ^2/2$. Using these and  \eqref{doby} we get
\begin{eqnarray*}
\rho( y',y_0+p'\ol y )&\le  &\rho (y', y_{n+1})+\rho (y_{n+1}, y_0+p_n\ol y)+
\rho (y_0+p_n\ol y, y_0+p'\ol y )\\
&\le &
\varepsilon^2/2+\varepsilon(1-\varepsilon)+\varepsilon^2/2=\varepsilon.
\end{eqnarray*}
Then from \eqref{eq:def:eps}, we have that $y'\in V$.  Passing to limit in \eqref{dob} we obtain that $x'\in x_0+\sigma B_X$, so $x'\in W$. Since $G$ has a closed graph, $(x',y')\in \grf G$. Finally, $(x',y')\in M'$.  This contradicts \eqref{contr}.

Hence the assumption is false and there must be some $n_0\in \N$ such that $p_{n_0}>1-\varepsilon$.

As all $t_i\in \,]0,\varepsilon[$, there is $m\le n_0$ such that
\begin{equation}\label{eq:m:def}
  p_m = \sum_{i=0}^{m}t_{i}\in\,]1-2\varepsilon,1-\varepsilon].
\end{equation}

It easily follows (see \eqref{dob})  that
$\|x_{m+1}-x_0\|\le\mu p_m\le \mu(1-\varepsilon) < \sigma$.
That is,
\begin{equation}\label{eq:x:m+1}
    x_{m+1}\in x_0+\sigma B_X.
\end{equation}
On the other hand,
$$
  \rho(y_{m+1}-y_0,\bar y) \le \rho(y_{m+1}-y_0,p_m\bar y) + \rho(p_m\bar y,\bar y),
$$
and from  \eqref{eq:y:smaller} and \eqref{eq:m:def} it follows that $\rho(y_{m+1}-y_0,p_m\bar y)\le \varepsilon p_m\le \varepsilon < \varepsilon_1$; while from  \eqref{eq:t-bar-y} and \eqref{eq:m:def} it follows that $\rho(p_m\bar y,\bar y) = (1-p_m) <2\varepsilon< \varepsilon_1$.
Therefore,
\begin{equation}\label{eq:y:m+1}
    y_{m+1}\in B_\rho ^\circ (y_0+\bar y;2\varepsilon_1).
\end{equation}
From \eqref{eq:x:m+1} and \eqref{eq:y:m+1} it follows that
%\begin{equation}\label{eq:claim:concl}
\[  \inf \rho\left(G(x_0+\sigma B_{X}),y_0+\bar y\right)<2\varepsilon_1.\]
%\end{equation}
Since $\varepsilon_1>0$ was arbitrary, the claim follows.
\end{proof}

\section{Proofs of the Main Results}\label{sec:proofs}

\setcounter{thm}{1}

Here we prove the results stated in Section~\ref{statements}.

\begin{thm}\label{thm:kornerstone}
  Let $(X,\|\cdot\|_n)$ and $(Y,|\cdot|_n)$ be Fr\'echet spaces and let
  $$
    F:X\rightrightarrows Y
  $$
  be a multi-valued map with closed graph.

  Let $U\subset X$ and $V\subset Y$ be open and such that $\gph F\cap (U\times V)\neq \emptyset$.

  Assume that for some $\mathbf{s}\in \mathbb{R}_+^\infty$ and some non-empty set $C\subset Y$ it holds that
 \begin{equation}\label{eq:main:th2:k2}
    D'F(x,y)(\Pi_{\mathbf{s}}(X))\supset  C,\quad \forall (x,y)\in \left( U\times V\right) \cap \grf F.
   \end{equation}
  Then for  all $(x,y)\in\grf F$ such that $x+\Pi_{\mathbf{s}}(X)\subset U$, and $[y,y+C]\subset V$,
  \begin{equation}\label{eq:main:th2:k3}
    \cl{F(x+\Pi_{\mathbf{s}}(X))}\supset y + C.
  \end{equation}
\end{thm}

\begin{proof}
  Fix $(x_0,y_0) \in (U\times V)\cap\grf F$ such that $x_0+  \Pi_{\s}(X){\subset }U$ and $[y_0,y_0{+} C]{\subset} V$.

  Obviously, it is enough to prove that
  \begin{equation}\label{eq:korn:1}
    y_0+\bar y\in\cl{F(x_0+\Pi_{\s}(X))}
  \end{equation}
  for each fixed $\bar y \in C$.

  We will apply Proposition~\ref{prop:4} to the Banach space $(X_{\s},\|\cdot\|_{\s})$, the metric space $(Y,\rho_Y)$ and the map $G:X_\s\rightrightarrows Y$ defined by
  $$
    G(x) := F(x_0 + x), \quad\forall x\in X_\s
  $$
  at the point $(0,y_0)\in\grf G$, while $\bar y$ will play the same role.

  Obviously, $[y_0,y_0{+}\bar y]= y_0+[0,1]\bar y\subset V$ and we will now translate and check the other assumptions of Proposition~\ref{prop:4}. In our case $\sigma$ will be equal to  $1$ and the set $W$ will be $W:=(U-x_0)\cap X_{\s}$. It is clear that $W$ is open in $X_{\s}$ (since $\|\cdot\|_{\s}$-topology is stronger than $\rho_X$-topology). Since $x_0+  \Pi_{\s}(X)\subset U$, we also have $W\supset  \Pi_{\s}(X) = B_{X_\s}$.

  Since $B_{X_{\s}}=\Pi_{\s}(X)$, we have by definition
  $$
    D'G(x,y)(B_{X_{\s}}) = D'F(x_0+x,y) (\Pi_{\s}(X)),\quad \forall (x,y)\in \grf G.
  $$
  If $y\in V$ and $x\in B_{X_{\s}}=\Pi_{\s}(X)$, so  $x_0+x\in U$, we have by \eqref{eq:main:th2:k2} that $D'F(x_0+x,y) (\Pi_{\s}(X)) \supset C$. Since $\bar y\in C$ we have, therefore,
  $$
   \bar y\in D'G(x,y)( B_{X_{\s}}),\quad\forall (x,y)\in \grf G\cap (W\times V).
  $$
 Since the assumptions of Proposition~\ref{prop:4} are satisfied, we apply it to get
  $$
   y_0+\bar y\in \cl{G( B_{X_{\s}})} = \cl{F(x_0+\Pi_\s(X))},
  $$
  which is \eqref{eq:korn:1} and the proof is completed. 
\end{proof}

\begin{cor}\label{cor:compact}
  Let the assumptions of Theorem~$\ref{thm:kornerstone}$ hold and let, moreover, $\Pi_{\mathbf{s}}(X)$ be compact.

 Then for  all $(x,y)\in\grf F$, $x+\Pi_{\s}(X)\subset  U$, and    $[y,y+C]\subset V$,
  \[
 {F(x+\Pi_{\mathbf{s}}(X))}\supset y + C.
  \]
\end{cor}

\begin{proof}
  Fix $\bar y \in C$.

  From the conclusion of Theorem~\ref{thm:kornerstone} (see \eqref{eq:main:th2:k3}) it follows that there exist $x_k\in\Pi_\s(X)$ and $y_k\in F(x+x_k)$ such that
  $
    y_k\to y+\bar y\mbox{ as }k\to\infty.
  $

  Since $\Pi_\s(X)$ is compact, there is a subsequence $(x_{k_m})_{m=1}^\infty$ such that\linebreak $\lim_{m\to\infty}x_{k_m}=\bar x\in \Pi_\s(X)$.

  Thus, $x+x_{k_m}\to x+\bar x$, as $m\to\infty$, and since $\grf F$ is closed, $(x+\bar x, y+\bar y)\in\grf F$. 
\end{proof}

\setcounter{thm}{6}

\begin{thm}  \label{thm:main-pi-surj}
  Let $(X,\|\cdot\|_n)$ and $(Y,|\cdot|_n)$ be Fr\'echet spaces and let
  $$
    F:X\rightrightarrows Y
  $$
  be a multi-valued map with closed graph.

  Let $U\subset X$ be open and let  $V\subset Y$ be open and convex. Assume that for some $\kappa>0$
  \begin{equation}\label{eq:main:th7:2}
    D'F(x,y)(\Pi_{\mathbf{s}}(X))\supset  \kappa\Pi_{\mathbf{s}}(Y),\quad \forall (x,y)\in \left( U\times V\right) \cap \grf F,\ \forall \mathbf{s}\in \mathbb{R}_+^\infty.
  \end{equation}
  Then $F$ is weakly $\Pi$-surjective on $(U,V)$ with constant $\kappa$.

If, moreover,  $\Pi_{\mathbf{s}}(X)$  are compact for all $\mathbf{s}\in \mathbb{R}_+^\infty$, then $F$ is $\Pi$-surjective on $(U,V)$ with constant $\kappa$.
\end{thm}

\begin{proof} There is nothing to prove if $\grf F\cap(U\times V)=\emptyset$.
	
Otherwise, take any $(x,y)\in \grf F\cap(U\times V)$ and then fix $\s\in\mathbb{R}_+^\infty$ such that $x+\Pi_{\s}(X)\subset U$.
 Set  $C:=\kappa\Pi_{\s}(Y)\cap\{V-y\}$, that is, $y+C = \{y+\kappa\Pi_{\s}(Y)\}\cap V$.

  Then, \eqref{eq:main:th7:2}  yields
\[
   D'F(x,y)(\Pi_{\s}(X))\supset \kappa\Pi_{\s}(Y) \supset  C,\quad \forall (x,y)\in \left( U\times V\right) \cap \grf F.
\]
  Since  $y+C\subset V$ and $tC\subset C$ for all $t\in [0,1]$, because $C$ is convex and $0\in C$, we have that $y+tC\subset V$ for all $t\in [0,1]$, hence $[y,y+C]\subset V$.
  We see that all assumptions of Theorem~\ref{thm:kornerstone} hold, and, therefore,
%   \begin{equation}\label{eq:main:th7:k3}
\[
    \cl{F(x+\Pi_{\s}(X))}\supset y + C.
    \]
%  \end{equation}

   That is,
  $$
  	\cl{F(x+ \Pi_{\s}(X))}\supset \{y + \kappa\Pi_{\s} (Y)\}\cap V
  $$
and $F$ is weakly $\Pi$-surjective on $(U,V)$ with constant $\kappa $.

    In the case when all $\Pi_\s(X)$, $\s\in\mathbb{R}_+^\infty$ are compact, we apply Corollary~\ref{cor:compact} in a similar way. 
\end{proof}

\setcounter{thm}{5}

\begin{thm}\label{thm:w-reg-m-reg}
  Let  $(X,\|\cdot\|_n)$ and $(Y,|\cdot|_n)$ be Fr\'echet spaces and let $U\subset X$ and $V\subset Y$ be nonempty and open.

  If the multi-valued map $F:X\rightrightarrows Y$ with closed graph is weakly $\Pi$-surjective on $(U, V)$ then there is $\theta >0$ such that
  \[
  \begin{array}{c}
  F(B_{\rho_X}(x;r))\supset B_{\rho_Y}(y;\theta r)\cap V,\\[10pt]
  \forall (x,y)\in\grf F\cap (U\times V),\ \forall r:\ 0<r< m_U(x),
  \end{array}
  \]
  where
  $$
  m_U(x) := \mathrm{dist}(x,X\setminus U).
  $$
  That is, $F$ is restrictedly $\gamma$-open at liner rate on $(U,V)$ for $\gamma (x)=m_U(x)$.
\end{thm}

\begin{proof}
  Let $F$ be weakly $\Pi$-surjective on $(U, V)$ with constant $\kappa$.

  Fix $\alpha > 0$ such that $\alpha < \min\{1,\kappa\}$.

Let $(x,y)\in \grf F$ and $r>0$ be such that $x\in U$ and $r<m_U(x)$. Let $y+v$, where $\rho_Y(0,v)\le \alpha r$, be an arbitrary point from $B_{\rho_Y}(y; \alpha r)\cap V$. Let $\s\in\mathbb{R}_+^\infty$ be defined as
$$
  s_n:=\alpha^{-1}\|v\|_n,\quad n=0,1,\ldots,
$$
so  $v\in\alpha\Pi_\s (Y)$.

Since $\alpha^{-1}>1$, from Lemma~\ref{lem:key} it follows that
$$
  \Pi_\s(X) = \alpha^{-1}\Pi_{\alpha\s}(X)\subset B_{\rho_X}(0;\alpha^{-1}|\alpha\s|),
$$
where $\alpha\s = (\|v\|_0,\|v\|_1,\|v\|_2,\ldots)$. Thus $|\alpha\s|=\rho_Y(0,v)\le \alpha r$. Therefore, $\Pi_\s(X)\subset B_{\rho_X}(0;r)$ and, in particular, $x+\Pi_\s(X)\subset U$. From the definition of weak  $\Pi$-surjectivity (after an obvious transformation) it follows that
$$
  \cl{F(B_{\rho_X}(x;r))}\supset \cl{F(x+\Pi_\s(X))}\supset \{y + \alpha\Pi_\s(Y)\}\cap V.
$$
So, $\cl{F(B_{\rho_X}(x;r))}{\supset} \{y {+} \alpha\Pi_\s(Y)\}\cap V$ and, in particular, $y{+}v{\in} \cl{F(B_{\rho_X}(x;r))}$. Since $v\in B_{\rho_Y}(0; \alpha r)\cap V$ was arbitrary, we get
$$
 \cl{F(B_{\rho_X}(x;r))}\supset B_{\rho_Y}(y;\alpha r)\cap V.
$$
Lemma~\ref{lem:density} completes the proof. 
\end{proof}

\setcounter{thm}{7}

\begin{thm}
	\label{thm:ci}
	Let $X=\cap_0^\infty(X_n,\|\cdot\|_n)$ and $Y=\cap_0^\infty (Y_n,|\cdot|_n)$ be compactly graded spaces. Let $f:X\to Y$ be continuous, strongly G\^ateaux differentiable  and such that $f(0)=0$.
	
	Assume that there are $c>0$ and $d\in \{0\}\cup\mathbb{N}$, such that for each $x\in X$ and $v\in Y$
	\begin{equation}
	\label{eq:f-prime:ld}
	\exists u\in X:\ f'(x)u=v\mbox{ and }
	\|u\|_n\le c|v|_{n+d},\quad\forall n\ge 0.
	\end{equation}	
	Then for each  $y\in Y$ there is $x\in X$ such that
	\begin{equation}
	\label{eq:f:ld}
	f(x)=y\mbox{ and }
	\|x\|_n\le c|y|_{n+d},\quad\forall n\ge 0.
	\end{equation}		
\end{thm}

\begin{proof}
	Since $Y=\cap_{n=d}^\infty (Y_n,|\cdot|_n)$, we can assume without loss of generality that $d=0$. Indeed, set $m:=n-d$, $n=d,d+1,\ldots$ and work with the new $m$-indexation in $Y$.
	
	Condition \eqref{eq:f-prime:ld} can be rewritten as
	$$
		f'(x)(\Pi_{\mathbf{s}}(X))\supset c^{-1}\Pi_{\mathbf{s}}(Y), \quad\forall x\in X,\ \forall \mathbf{s}\in \mathbb{R}_+^\infty.
	$$
	Clearly, for strongly G\^ateaux differentiable function $f$,
	$$
			c^{-1}\Pi_{\mathbf{s}}(Y)\subset f'(x)(\Pi_{\mathbf{s}}(X)) \subset D'f(x,f(x))(\Pi_{\mathbf{s}}(X)).
	$$
	Also, $\grf f$ is closed, because $f$ is continuous.
	
	From Proposition~\ref{pro:ci-comp} and Theorem~\ref{thm:main-pi-surj} it follows that $f$ is $\Pi$-surjective on $(X,Y)$ with constant $c^{-1}$, which  implies \eqref{eq:f:ld}. 
\end{proof}

\setcounter{section}{5}
\section{Conclusions}
	In this article we demonstrate new approach towards the Nash-Moser Theorem.
	
	We define Banach spaces that fit the structure of the constraints and this way take the problem from more geometrical angle.
	
	We tackle the problem in multi-valued setting and reveal its relationship to map regularity.
	
	The method for proving this regularity, which is based on an abstract iteration scheme, is also new.
	
	We believe that the ideas here presented can be developed in several directions.

\bigskip\noindent
\textbf{Acknowledgments:}
We  express our sincere gratitude  to  Prof. Asen Dontchev for constantly pushing us towards this subject.\\
We are deeply thankful to the anonymous referee for pointing some inaccuracies and incompleteness in the initial version.\\
Special thanks are due to Prof. Radek Cibulka for reading the manuscript very carefully and pointing out that in metric space the definition of differentiability we use is more demanding than standard G\^ateaux differentiability.\\
Research is supported by the Bulgarian National  Fund for Scientific Research, contract KP-06-H22/4.

\end{document}